\documentclass[11pt]{amsart}
\usepackage{graphicx}
\usepackage{a4wide}
\usepackage{amssymb}
\usepackage{amsthm}
\usepackage{amsmath}
\usepackage{amscd} \usepackage{verbatim}
\textheight8.5in \textwidth6.7in \numberwithin{equation}{section}
\usepackage{hyperref}
\newtheorem{theorem}{Theorem}[section]

\newtheorem{lemma}[theorem]{Lemma}
\newtheorem{proposition}[theorem]{Proposition}

\theoremstyle{definition}

\newtheorem{example}[theorem]{Example}
\theoremstyle{remark}
\newtheorem*{remark}{Remark}

\newcommand{\R}{\mathbb{R}}

\newcommand{\N}{\mathbb{N}}

\newcommand{\htis}{\widehat{T}_{r,N}}
\newcommand{\diff}{\widehat{T}_{r,N}(n) - \widehat{T}_{N-r,N}(n)}

\newcommand{\That}{\widehat{T}}

\renewcommand{\Re}{\operatorname{Re}}
\renewcommand{\Im}{\operatorname{Im}}

\begin{document}

\title[The number of parts in residue classes]
{On the number of parts of integer partitions lying in given residue classes}
\author{Olivia Beckwith and Michael H. Mertens}
\address{Department of Mathematics and Computer Science,
Emory University, Atlanta, Georgia 30322}
\email{olivia.dorothea.beckwith@emory.edu}
\email{mmerten@emory.edu}

\begin{abstract}
Improving upon previous work \cite{BeMe} on the subject, we use Wright's Circle Method to derive an asymptotic formula for the number of parts in all partitions of an integer $n$ that are in any given arithmetic progression.
\end{abstract}

\maketitle

\section{Introduction and statement of results}\label{sec:intro}
The study of asymptotics for partitions is by now almost a century old. It originated from seminal work by Hardy and Ramanujan \cite{HR}, who invented and employed the \emph{Circle Method}, a by now vital tool in Analytic Number Theory, to obtain their famous asymptotic formula for the partition numbers $p(n)$ as $n\rightarrow\infty$,
\[p(n)\sim \frac{1}{4n\sqrt{3}}\exp\left(\pi\sqrt{\frac{2n}{3}}\right).\]
Some 20 years later, Rademacher \cite{Rademacher} refined the methods of Hardy and Ramanujan to obtain an exact formula for the partition numbers in terms of an infinite sum of Kloosterman sums weighted by Bessel functions. Both methods rely heavily on the fact that the generating function for the partition numbers is essentially a modular form, namely the reciprocal of the Dedekind eta function,
\[\eta(\tau)=e^{\frac{\pi i\tau}{12}}\prod_{n=1}^\infty \left(1-e^{2\pi in\tau}\right),\quad \Im(\tau)>0.\]

Several authors \cite{DaSa, DaSaSz, DaSaSz2, ErdosLehner}, have also addressed the question how many parts with certain properties, e.g. congruence conditions, there are asymptotically in a given partition of a positive integer $n$. This also the object of the present paper. 

Before proceeding, let us introduce and fix some notation. For a \emph{partition} $\lambda = (\lambda_0, \dots ,\lambda_k)$, i.e. a non-increasing sequence of positive integers, we let 
\begin{equation}
T_{r,N} (\lambda) := | \{ \lambda_j : \lambda_j \equiv r \pmod{N} \} |
\end{equation}
denote the number of parts in $\lambda$ which are congruent to $r$ modulo $N$. We also set 
\[|\lambda|:=\sum_{j=0}^k \lambda_j.\]
For a positive integer $n$ we then define 
\begin{equation}
 \htis (n) := \sum_{|\lambda| = n} T_{r,N} (\lambda),
\end{equation} 
where the summation runs over all partitions of size $n$. The quantity $\htis(n)$ counts the number of parts congruent to $r  \pmod{N}$ in all partitions of $n$. For example, all partitions of $5$ are
\[(5),\ (4,1),\ (3,2),\ (3,1,1),\ (2,2,1),\ (2,1,1,1),\ (1,1,1,1,1),\]
hence $\That_{1,3}(5)=13$ and $\That_{2,3}(5)=5$.

In \cite{BeMe}, the authors have shown the following asymptotic formula for differences of these quantities.
\begin{theorem}\label{thm:th2}
Let $r,N$ be coprime positive integers with $N\geq 3$ and $1\leq r< N$. Then, as $n\rightarrow\infty$, we have
\begin{align*}
&\diff = \frac{1}{2 \sqrt{2}  N} \cot \left( \frac{\pi r}{N} \right) \frac{e^{ \left(\pi \sqrt{ \frac{2}{3} \left( n - \frac{1}{24} \right)} \right)}}{\sqrt{ \left( n - \frac{1}{24} \right)} } \\
  -&\frac{1}{4\sqrt{3}\varphi(N)}\sum_{\psi(-1) = -1} \psi( r ')L(0,\psi)\frac{e^{ \left(\pi \sqrt{ \frac{2}{3} \left( n - \frac{1}{24} \right)} \right)}}{ n - \frac{1}{24}  }+ O\left( n^2 e^{ \left( \frac{\pi}{2} \sqrt{ \frac{2}{3} \left( n - \frac{1}{24} \right)} \right)} \right) ,
\end{align*}
where $\psi$ runs through all odd Dirichlet characters modulo $N$, $L(s,\psi)$ denotes the Dirichlet $L$-series associated to $\psi$, $r'$ denotes the multiplicative inverse of $r$ modulo $N$, and $\varphi$ denotes Euler's totient function.
\end{theorem}
\begin{remark}
The statement of Theorem \ref{thm:th2} is slightly different in shape but equivalent to that of Theorem 1.1 in \cite{BeMe}. This representation makes it clear that for $n$ sufficiently large we have the inequality
\[\That_{r,N}(n)\geq \That_{N-r,N}(n),\]
whenever $1\leq r\leq \tfrac N2$ and $\gcd(N,r)=1$.
\end{remark}

The proof of Theorem \ref{thm:th2} relies again of the modularity for the generating function of the quantity $\diff$, which turns out to be given by the quotient of an explicit weight $1$ Eisenstein series for the principal congruence subgroup $\Gamma(N)$ by the Dedekind eta function. The generating function of the individual terms $\That_{r,N}(n)$ however has absolutely no modularity properties so that an analogous approach wouldn't work. But using a different version of the Circle Method originally due to E. M. Wright \cite{Wright}\footnote{This method has been rediscovered and used in various situations mainly by K. Bringmann and K. Mahlburg , see e.g. \cite{BM1,BM2}, and their collaborators} in a convenient formulation given in \cite{RhoadesNgo}, as well as the famous Euler-Maclaurin summation formula we can prove the following result.

\begin{theorem}\label{thm:main}
For fixed numbers $r,N\in\N$, $r\leq N$, we have the asymptotic
\begin{align*}
\That_{r,N} (n) = e^{\pi \sqrt{ \frac{2n}{3}}} n^{- \frac{1}{2}} \frac{1}{4 \pi N \sqrt{2}} \left[  \log n -  \log \left( \frac{\pi^2}{6} \right)\right. -2  \left(\psi\right.&\left.\left(\frac rN\right)+ \log N\right)\\
&\left. + O\left(n^{- \frac{1}{2}} \log n\right)\right],
\end{align*}
where $\psi(z)=\frac{\Gamma'(z)}{\Gamma(z)}$ denotes Euler's digamma function, as $n\rightarrow\infty$.
\end{theorem}
\begin{remark}
\begin{enumerate}
\item Plugging in $r=N$ in Theorem \ref{thm:main} and using the well-known fact that $\psi(1)$ equals the negative of the Euler-Mascheroni constant $\gamma_E=0.577215664...$ (see e.g. equation 5.4.12 in \cite{NIST}) recovers Theorem 1.3 in \cite{BeMe} which was proven using results about period functions of Maa{\ss}-Eisenstein series.
\item Since the digamma function $\psi(x)$ is monotonically increasing and negative for real arguments $x\leq 1$, we see at once that for $1\leq r< s\leq N$ we have the inequality
\[\That_{r,N}(n)\geq \That_{s,N}(n)\]
for all sufficiently large $N$.
\end{enumerate}
\end{remark}
\begin{example}
Let $N=3$. We want to illustrate the asymptotic formula from Theorem \ref{thm:main} for $\That_{r,3}(n)$ for $r=1,2,3$. Let $Q_{r}(n)$ denote the quotient of $\That_{r,3}(n)$ by the respective main term. Table \ref{numerics1} gives the numerical values (decimal expansions are truncated, not rounded).
\begin{table}[h!]
\begin{tabular}{|c|c|c|c|c|c|}
\hline 
$n$ & $10$ & $100$ & $1,000$ & $10,000$ & $100,000$  \\ 
\hline
$Q_1(n)$ & 0.982155 & 0.992241 & 0.997608 & 0.999273 & 0.999778  \\ 
\hline 
$Q_2(n)$ & 1.149645 & 1.017114 & 1.003063 & 1.000592 & 1.000115  \\ 
\hline 
$Q_3(n)$ & 1.792248 & 1.067095 & 1.011771 & 1.002470 & 1.000563  \\ 
\hline 
\end{tabular}
\vspace{0.5cm}
\caption{Numerics for Theorem \ref{thm:main}}
\label{numerics1}
\end{table}
\end{example}

The rest of this paper is organized as follows. In Section \ref{secPrelim} we collect some preliminary results on Euler-Maclaurin summation and Wright's Circle Method, and in Section \ref{secProof}, we prove Theorem \ref{thm:main}.

\section*{Acknowledgements}
The auhors would like to thank Kathrin Bringmann for drawing the reference \cite{Zagier} to their attention and Ken Ono for helpful discussions.

\section{Preliminaries}\label{secPrelim}
\subsection{The Euler-Maclaurin formula and asymptotic expansions}
In this subsection we recall a maybe not too widely known version of the Euler-Maclaurin summation formula. Throughout the paper we use the common notation 
$$h(t) \sim \sum_{k = -1}^{\infty} a_k t^k,\quad (t \to 0)$$ 
for asymptotic expansions in its strong sense, meaning that for every $M \in\N_0$ we have 
$$h(t) - \sum_{k=-1}^{M-1} a_k t^k=O(t^M),\quad (t\to 0).$$ 

In the following, we will often encounter \emph{Bernoulli polynomials} $B_n(x)$ for $n$ a non-negative integer, which can be defined via their generating function
\begin{equation*}
\sum_{n=0}^\infty B_n(x)\frac{t^n}{n!}:=\frac{te^{xt}}{e^t-1},\quad |t|<2\pi,
\end{equation*}
and the \emph{Bernoulli numbers} $B_n:=B_n(0)$. 

The Bernoulli polynomials are also given explicitly in terms of the Bernoulli numbers as follows:
\begin{equation}\label{equation:bernoulli}
B_n(x) = \sum_{k=0}^n {{n}\choose{k}} B_{n-k} x^k
\end{equation}
and they satisfy the following relations,
\begin{equation}\label{eqBernoulliSum}
B_n(x+1)-B_n(x)=nx^{n-1}
\end{equation}
and 
\begin{equation}\label{eqBernoulliSym}
B_n(1-x)=(-1)^nB_n(x),
\end{equation}
see e.g. equations 24.4.1 and 24.4.3 in \cite{NIST}.

In Proposition 3 in \cite{Zagier}, Zagier gives the following formula, which we use (see also \cite[Proposition A.1]{BM1}). He proves a slightly more restrictive version of the following theorem, but states the version displayed here\footnote{The equation he gives (see eq. (44) in \cite{Zagier}), however, contains a slight typo, the + sign in front of the sum should be a - sign, as also pointed out in \cite{BM1}.}. The reader is also referred to formula 23.1.32 in \cite{AbrSteg} as well as \cite{Lampret} and the references therein.
\begin{proposition}\label{propZagier}
Let $f$ be a $C^{\infty}$ function on the positive real line which has an asymptotic expansion $f(t) \sim \sum_{n=0}^{\infty} b_n t^n$ as $t \to 0$, and satisfies the property that it and all of its derivatives are of rapid decay at infinity. Then we have the asymptotic expansion  
$$
\sum_{m = 0}^{\infty} f( (m + a) t) \sim \frac{ 1}{t}\int_0^{\infty} f(t) dt - \sum_{n=0}^{\infty} b_n \frac{B_{n+1} (a) }{n+1} t^n, \quad (t \to 0).
$$
for every $a>0$.
\end{proposition}
\begin{remark}
An inspection of the proof of Proposition \ref{propZagier} shows that the given asymptotic expansion is actually valid whenever $t$ is a complex variable with $|\arg(t)|< \frac{\pi}{2}-\delta$ for some $\delta>0$ provided that $f^{(n)}(e^{i\theta}x)$ is of rapid decay for real $x\to\infty$, $|\theta|<\frac{\pi}{2}-\delta$ and all non-negative integers $n$.
\end{remark}
For the convenience of the reader, we give a proof of Proposition \ref{propZagier}.
\begin{proof}
For some $t$, Let $g(x) := f((a + x) t)$. Note that $g$ is still smooth and has derivatives of rapid decay at infinity. Applying the first formula on page 13 of \cite{Zagier} to $g(x)$ gives us the following:
\begin{align*}
\sum_{m=1}^{\infty} f((m+a)t)  = \frac{ \int_{at}^{\infty} f(x) dx}{t} + &\sum_{n=1}^{N-1} \frac{(-1)^n B_{n+1} }{(n+1)!}  
g^{(n)}(0)\\
& \qquad \qquad - (-1)^N \int_0^{\infty} g^{(N)} (x) \frac{\widehat{B}_N (x)}{N!} dx.
\end{align*}

where $\widehat{B}_N(x) := B_n(x - \lfloor x \rfloor)$. 

We notice that the first term is given by 
$$   \frac{ \int_{at}^{\infty} f(x) dx}{t} = \frac{ \int_0^{\infty} f(x) dx}{t} - \frac{ \int_0^{at} f(x) dx}{t}$$

Using the asymptotic expansion for $f$, we have 
\begin{equation*}
\frac{ \int_0^{at} f(x) dx}{t} = \sum_{n=0}^{N-1} \frac{b_n a^{n+1} t^n }{n+1} + O(t^N), \quad ( t \to 0).
\end{equation*}

We notice that the last integral is $O(t^N)$ as $t \to 0$, since
\begin{equation*}
 - (-1)^N \int_0^{\infty} g^{(N)} (x) \frac{\widehat{B}_N (x)}{N!} dx = (- t)^{N-1} \int_0^{\infty} f^{(N)} (x + a) \frac{\widehat{B}_N(\frac{x}{t})}{N!} dx, 
\end{equation*}
and since $\widehat{B}_N (x)$ is bounded and $f^{(N)}$ is of rapid decay. 

Now we consider the second sum. We have $g^{(n) }(0) = t^n f^{(n)} (at)$, which has the following expansion:
$$
g^{(n) }(0) = t^n \left( \sum_{m=0}^{N-1 - n} b_{m+n} \frac{(m+n)!}{m!} a^m t^m + O(t^{N-n}) \right), \quad (t \to 0).
$$ 

Substituting this formula and switching the order of summation, we find the asymptotic expansion as $t \to 0$ for the middle sum:
\begin{equation*}
\begin{aligned}
&\sum_{n=1}^{N-1} \frac{(-1)^n B_{n+1} }{(n+1)!} g^{(n)}(0)  \\
=& \sum_{n=1}^{N-1} \frac{(-1)^n B_{n+1} }{(n+1)!}  t^n \left( \sum_{m=0}^{N-1 - n} b_{m+n} \frac{(m+n)!}{m!} a^m t^m + O(t^{N-n}) \right), \quad (t \to 0) \\
=& \sum_{k=0}^{N-1} \frac{b_k t^k}{k+1} \sum_{n=0}^k (-1)^{n} B_{n+1} a^{k-n} {{k+1}\choose{n+1}} + O(t^N), \quad (t \to 0).
\end{aligned}
\end{equation*}

Now we can put everything together to obtain (using the shorthand $I_f:=\int_0^\infty f(x)dx$) that
\begin{align*}
&\sum_{m=0}^\infty f((m+a)t)=f(at)+\sum_{m=1}^\infty f((m+a)t)\\ 
=& \sum_{n=0}^{N-1} b_na^nt^n +\frac{I_f}{t}-\sum_{n=0}^{N-1} \frac{b_na^{n+1}}{n+1}t^n +\sum_{n=0}^{N-1} \frac{b_nt^n}{n+1}\\
&\qquad\qquad\qquad\qquad\qquad\qquad\times \left[\sum_{k=0}^n (-1)^k B_{k+1}a^{n-k} {{n+1} \choose {k+1}}\right]+O(t^N)\\
						   =& \frac{I_f}{t} + \sum_{n=0}^{N-1} b_na^nt^n + \sum_{n=0}^{N-1} \frac{b_n}{n+1} \left[\sum_{k=0}^{n+1} {{n+1} \choose k} B_k(-a)^{n+1-k} \right](-t)^n+O(t^N).
\end{align*}
By \eqref{equation:bernoulli}, we recognize the sum in square brackets as the Bernoulli polynomial $B_{n+1}(-a)$. Then using \eqref{eqBernoulliSum} and \eqref{eqBernoulliSym}, one easily sees that the coefficient of $t^n$ ($n\geq 0$) in the above expansion is given by $-\frac{B_{n+1}(a)}{n+1}$, which is what we claimed.
\end{proof}
\subsection{Wright's Circle Method}\label{secWright}
In this section, we briefly recall two propositions from \cite{RhoadesNgo}, based on Wright's version of the Circle Method \cite{BM1, Wright}, that allow to obtain asymptotic results for products of functions in a fairly general setting. 

Suppose $\xi(q)$ and $L(q)$ are analytic functions for complex arguments $|q|<1$ and $q\notin\R_{\leq 0}$, such that 
$$\xi(q)L(q)=:\sum_{n=0}^\infty a(n)q^n$$
is analytic for $|q|<1$. Further assume the following hypotheses, where $0<\delta<\tfrac\pi 2$ and $c>0$ are fixed constants.
\begin{enumerate}
\item\label{hypo1} As $t\rightarrow 0$ in the cone $|\arg (t)|<\tfrac \pi 2-\delta$ and $|\Im (t)|\leq\pi$ we have, for some $B\in\R$, either
\begin{equation}\label{LpolyMajor}
L(e^{-t})=t^{-B}\left(\sum\limits_{\ell=0}^{k-1}\alpha_\ell t^\ell+O_\delta(t^k)\right),
\end{equation}
in which case we say that $L$ is \emph{of polynomial type near $1$},
or
\begin{equation}\label{LlogMajor}
L(e^{-t})=\frac{\log t}{t^B}\left(\sum\limits_{\ell=0}^{k-1}\alpha_\ell t^\ell+O_\delta(t^k)\right),
\end{equation}
in which case we call $L$ \emph{of logarithmic type near $1$}. 
\item\label{hypo2} As $t\rightarrow 0$ in the cone $|\arg (t)|<\tfrac \pi 2-\delta$ and $|\Im (t)|\leq \pi$ we have
\begin{equation}\label{xiMajor}
\xi(e^{-t})=t^\beta e^{\frac{c^2}{t}}\left(1+O_\delta(e^{-\frac \gamma t})\right)
\end{equation}
for real constants $\beta\geq 0$ and $\gamma>c^2$.
\item\label{hypo3} As $t\rightarrow 0$ in the cone $\tfrac \pi 2-\delta\leq |\arg (t)|<\tfrac \pi 2$ and $|\Im (t)|\leq \pi$ one has
\begin{equation}\label{Lminor}
|L(e^{-t})|\ll_\delta |t|^{-C},
\end{equation}
where $C=C(\delta)>0$.
\item\label{hypo4} As $t\rightarrow 0$ in the cone $\tfrac \pi 2-\delta\leq |\arg (t)|<\tfrac \pi 2$ and $|\Im (t)|\leq \pi$ one has
\begin{equation}\label{ximinor}
|\xi(e^{-t})|\ll_\delta\xi(|e^{-t}|)e^{-K\Re\left(\frac 1{t}\right)},
\end{equation}
where $K=K(\delta)>0$.
\end{enumerate} 

These hypotheses in \eqref{LpolyMajor}--\eqref{xiMajor} ensure the asymptotics of $L$ and $\xi$ on the so-called \emph{major arc}, those in \eqref{Lminor}--\eqref{ximinor} their asymptotics on the so-called \emph{minor arc} of the unit circle. 

For our purposes, we require the following two propositions (see Propositions 1.8 and 1.10 in \cite{RhoadesNgo}).
\begin{proposition}\label{Wrightpoly}
Suppose the hypotheses (1)--(4) are satisfied and that $L$ has polynomial type near $1$. Then there is an asymptotic expansion
\begin{equation}\label{asympoly}
a(n)=e^{2c\sqrt{n}}n^{\frac 14(2B-2\beta-3)}\left(\sum\limits_{r=0}^{M-1}p_rn^{-\frac r2}+O(n^{-\frac M2})\right),
\end{equation}
where
\begin{equation}
p_r=\sum\limits_{s=0}^r \alpha_sw_{s,r-s}
\end{equation}
with $\alpha_s$ as in \eqref{LpolyMajor} and 
\begin{equation}
w_{s,r}=\frac{c^{s+\beta-B+\frac 12}}{(-4c)^r2\pi^\frac 12}\cdot\frac{\Gamma\left(s+\beta-B+r+\frac 32\right)}{r!\Gamma\left(s+\beta-B-r+\frac 32\right)}
\end{equation}
for the coefficients $a(n)$ of $\xi(q)L(q)$ as $n\rightarrow\infty$.
\end{proposition}

Note that Proposition \ref{Wrightpoly} is originally due to Wright \cite{Wright}.

\begin{proposition}\label{Wrightlog}
Suppose hypotheses (1)--(4) are satisfied and that $L$ has logarithmic type near $1$ such that $B-\beta=\tfrac 12$, with $B$ and $\beta$ as in equations \eqref{LlogMajor} and \eqref{xiMajor} respectively. Then we have 
\[a(n)=-e^{2c\sqrt{n}}n^{-\frac 12}\frac{\alpha_0}{4\pi^\frac 12}\left(\log n-2\log c+O(n^{-\frac 12}\log n)\right)\]
as $n\rightarrow\infty$.
\end{proposition}
\subsection{A Preliminary Lemma}
Here we prove a preliminary result that will be the key step towards the proof of Theorem \ref{thm:main}. For the rest of the paper, let
\[f(t):=\frac{1}{e^t-1},\quad \Re(t)>0\]
\begin{lemma}\label{thm:lem1}
Let $r,N$ be greater than zero. Then we have for $|\arg(t)|< \frac\pi2-\delta$ for some $0<\delta<\tfrac\pi2$: 
$$
\sum_{m = 0}^{\infty} f\left(\left(m+ \frac{r}{N} \right) t\right) \sim - \frac{ \log (t)+\psi\left(\frac rN\right)}{t} + O(\log t),\quad (t\to 0),
$$
where $\psi(z)=\frac{\Gamma'(z)}{\Gamma(z)}$ denotes Euler's digamma function.
\end{lemma}
\begin{proof}
The proof follows the procedure described in the remarks following the proof of \cite[Proposition 3]{Zagier}.

From the definition of the Bernoulli numbers we have 
$$ 
f(t) = \sum_{k = 0}^{\infty} \frac{ B_k}{k!} t^{k-1}.
$$

Let $f^*(t) := f(t) - t^{-1} e^{-t}$. Then we have for $\Re(t)>0$,
\begin{equation}\label{equation:eq1}
\sum_{m = 0}^{\infty} f\left(\left(m+ \frac{r}{N} \right) t\right) = \sum_{m = 0}^{\infty} \frac{1}{ \left( m + \frac{r}{N} \right)t} e^{-\left(m + \frac{r}{N}\right)t} + \sum_{m = 0}^{\infty} f^* \left( \left(m + \frac{r}{N}\right)t\right). 
\end{equation}
To find the asymptotic expansion of the second sum in the right hand side of \ref{equation:eq1}, we note that
$$f^* (t) \sim \sum_{k=0}^{\infty} b_k t^k,\quad (t\to 0),$$
where
$$
b_k := (B_{k+1}  -1) \frac{1}{(k + 1)!}.
$$

Then it follows from Proposition \ref{propZagier} that we have
$$
\sum_{m = 0}^{\infty} f^* \left( \left(m + \frac{r}{N} \right)t\right) \sim \frac{ \int_{0}^{\infty} f^*(t) dt}{t} - \sum_{n = 0}^{\infty} b_n \frac{B_{n+1} \left(\frac{r}{N}\right)}{n+1} t^n , \quad (t\to 0).
$$

Next we consider the first sum of the right hand side of \ref{equation:eq1}. 

First, since
 $$ \frac{1}{m+ \frac{r}{N}} = \left( \frac{1}{m+ \frac{r}{N}} - \frac{1}{m} \right) + \frac{1}{m},$$
we have the following:
$$
\sum_{m = 0}^{\infty} \frac{1}{  m + \frac{r}{N} } e^{-mt}= \frac{N}{r}+  \sum_{m=1}^{\infty} \frac{e^{-mt}}{m} - \sum_{m=1}^{\infty} \frac{ \frac{r}{N}}{m (m + \frac{r}{N})} e^{-mt}.
$$

The first sum is equal to $ - \log ( 1- e^{-t}) $. Differentiating once with respect to $t$, we see that the following holds:
$$
- \log\left( 1 - e^{-t}\right) \sim \log\left( \frac{1}{t}\right) -  \sum_{n = 1}^{\infty} \frac{B_n}{n \cdot n!} t^n ,\quad (t\to 0). 
 $$
 
The second sum is absolutely and uniformly convergent for $\Re(t) \ge 0$ and we have 
\[\lim\limits_{t\rightarrow 0}\sum_{m=1}^{\infty} \frac{ 1}{m (m + \frac{r}{N})} e^{-mt}= \sum_{m=1}^{\infty} \frac{ 1}{m (m + \frac{r}{N})}=\frac Nr\left(\gamma_E+\psi\left(\frac rN\right)\right)+\frac{N^2}{r^2}\]
by equation 5.7.6 in \cite{NIST}\nocite{Olver}, where $\gamma_E$ denotes the Euler-Mascheroni constant. We have that in fact
\[\sum_{m=1}^{\infty} \frac{ 1}{m (m + \frac{r}{N})} e^{-mt}=\frac Nr\left(\gamma_E+\psi\left(\frac rN\right)\right)+\frac{N^2}{r^2}+O(t\log t), \quad (t\rightarrow 0),\]
or, equivalently, that
\[\lim\limits_{t\rightarrow 0} \sum_{m=1}^{\infty} \frac{ 1}{m (m + \frac{r}{N})} \frac{e^{-mt}-1}{t\log t}\]
exists, which can easily be seen by applying de l'H\^opital's Rule (since the series is still absolutely and locally uniformly convergent for $\Re(t)>0$, we can differentiate each summand).

Assembling all of this gives the following:

\begin{align*}
&\sum_{m = 0}^{\infty} f\left(\left(m+ \frac{r}{N} \right) t\right)\\
 \sim& \frac{e^{- \frac{r}{N} t}}{t} \left( \frac{N}{r} + \log\left( \frac{1}{t} \right) -  \sum_{n = 1}^{\infty} \frac{B_n}{n \cdot n!} t^n + \sum_{m=1}^{\infty} \frac{ - \frac{r}{N}}{m (m + \frac{r}{N})} e^{-mt} \right)  \\ 
&\qquad\qquad\qquad\qquad\qquad+ \frac{ \int_{0}^{\infty} f^*(t) dt}{t} - \sum_{n = 0}^{\infty} b_n \frac{B_{n+1} (\frac{r}{N})}{n+1} t^n ,\quad (t\to 0).
\end{align*}

Simplifying, we have
$$
\sum_{m = 0}^{\infty} f\left(\left(m+ \frac{r}{N} \right) t\right) \sim - \frac{ \log (t)+\psi\left(\frac rN\right)}{t} + O(\log t) , \quad (t\to 0).
$$
Here we used the fact (see eq. 5.9.18 in \cite{NIST}) that 
\[\int_0^\infty f^*(t)dt=\int_0^\infty \frac{1}{e^t-1} -  \frac{e^{-t}}{t}dt=\gamma_E.\]
\end{proof}

\section{Proof of Theorem \ref{thm:main}}\label{secProof}
Now we prove the Theorem \ref{thm:main}. 
\begin{proof}
By Lemma 2.1 in \cite{BeMe} we have that
\begin{equation*}
\sum_{n=1}^{\infty} \widehat{T}_{r,N} (n) q^n = \left( \prod_{n \ge 1} \frac{1}{1 - q^n} \right) \left( \sum_{n = 1}^{\infty} \left( \sum_{\substack{d |n \\ d \equiv r \pmod{N}}} q^n \right) \right) 
\end{equation*}
Letting $f$ be as defined in the previous section and setting $q := e^{-t}$, we simplify this as follows:

\begin{align*} 
\sum_{n=1}^{\infty} \widehat{T}_{r,N} (n) q^n  &=  \left( \prod_{n \ge 1} \frac{1}{1 - q^n} \right) \sum_{m \equiv r \pmod{N}} \frac{q^m}{1 - q^m} \\
&=  \left( \prod_{n \ge 1} \frac{1}{1 - q^n} \right) \sum_{m=0}^{\infty} f\left( \left(m + \frac{r}{N}\right) Nt\right).
\end{align*}

Now we wish to apply the method outlined in Section \ref{secWright} with $\xi (e^{-t}) := \frac{(2 \pi)^{\frac{1}{2}}}{ \eta ( \frac{it}{2 \pi})}$ and $L(e^{-t}) := (2\pi)^{-\frac 12}q^{\frac{1}{24}} \sum_{m=0}^{\infty} f\left( \left(m + \frac{r}{N}\right) Nt\right)$. The function $\xi (q)$ is well-known to satisfies hypotheses \ref{hypo2} and \ref{hypo4} in Section \ref{secWright} with $c^2 = \frac{\pi^2}{6}$, $\beta = \frac{1}{2}$, and  $\gamma = 4 \pi^2$ (see e.g. \cite[Section 1.4, p. 280]{HR} or \cite[pp. 112--113]{Wright}). 

From Lemma \ref{thm:lem1} we see that $L(q)$ satisfies hypothesis \ref{hypo1}. By the straightforward estimate 
\[\left|\sum_{m\equiv r\pmod N} \frac{q^m}{1 - q^m}\right|\leq \sum_{m=1}^\infty \frac{|q|^m}{1 - |q|^m}\]
and by a well-known and essentially straight-forward estimate, carried out for example in the proof of Corollary 4.5 in \cite{RhoadesNgo}, we see that
\[\sum_{m\equiv r\pmod N} \frac{q^m}{1 - q^m}\ll_\delta t^{-\frac 32}\]
in the bounded cone $\frac\pi2-\delta\leq |\arg(t)|<\frac\pi2$ and $|\Im(t)|\leq \pi$, so that $L(q)$ also satisfies hypothesis \ref{hypo3} so that we can apply Propositions \ref{Wrightpoly} and \ref{Wrightlog}.

To be more precise, the asymptotic expansion as $t \to 0$ has both a polynomial and logarithmic asymptotic component. The logarithmic component is as follows:
$$
L_1 (e^{-t} ) = -\frac{\log (t)}{(2 \pi)^{\frac{1}{2}} N t} (1 + O(1)),
$$
to which we can apply Proposition \ref{Wrightlog}, with $B=1$ and $\alpha_0 = - \frac{1}{(2 \pi)^{\frac{1}{2}} N}$.

For the polynomial part, we have the following:
$$
L_2 (e^{-t}) =  - \frac{\log N}{Nt (2 \pi)^{\frac{1}{2}}} - \frac{\psi\left(\frac rN\right)}{Nt (2 \pi)^{\frac{1}{2}}}.
$$
For $L_2$, we apply Proposition \ref{Wrightpoly}, with $B=1$, $M=1$, $\alpha_0 = -\frac{\left(\psi\left(\frac{N}{r}\right)+\log N\right)}{N (2 \pi)^{\frac{1}{2}}} $. 
Putting these results together completes the proof.
\end{proof}

\end{document}